\documentclass[a4paper,11pt,leqno,english]{smfart}
\usepackage{aeguill}
\usepackage{enumerate}
\usepackage{amssymb,amsmath,latexsym,amsthm}
\usepackage[T1]{fontenc}
\usepackage{geometry}
\usepackage{url}
\usepackage[frenchb, main=english]{babel}
\usepackage[utf8]{inputenc}
\usepackage{mathrsfs}
\usepackage{xcolor}
\usepackage{comment}
\definecolor{violet}{rgb}{0.0,0.2,0.7}
\definecolor{rouge2}{rgb}{0.8,0.0,0.2}
\usepackage{tikz}
\tikzset{
  closed/.style = {decoration = {markings, mark = at position 0.5 with { \node[transform shape, xscale = .8, yscale=.4] {/}; } }, postaction = {decorate} },
  open/.style = {decoration = {markings, mark = at position 0.5 with { \node[transform shape, scale = .7] {$\circ$}; } }, postaction = {decorate} }
}
\usepackage{empheq}
\usepackage{tikz-cd}
\usetikzlibrary{matrix,arrows,decorations.pathmorphing}
\usepackage{hyperref}
\usepackage{mathpazo}
\hypersetup{
    bookmarks=true,         
    unicode=false,          
    pdftoolbar=true,        
    pdfmenubar=true,        
    pdffitwindow=false,     
    pdfstartview={FitH},    
    pdftitle={},    
    pdfauthor={},     
    colorlinks=true,       
   linkcolor=rouge2,          
    citecolor=violet,        
    filecolor=black,      
    urlcolor=cyan}           
\usepackage{enumitem}
\usepackage{appendix}

 \theoremstyle{plain}    
 \newtheorem{thm}{Theorem}[section]
\theoremstyle{plain} 
\newtheorem{bigthm}{Theorem}
\newtheorem{bigcoro}[bigthm]{Corollary}

 \numberwithin{equation}{section} 
 \numberwithin{figure}{section} 
 \newtheorem{cor}[thm]{Corollary} 
 \theoremstyle{plain}    
 \newtheorem{prop}[thm]{Proposition} 
 \theoremstyle{plain}    
 \newtheorem{lem}[thm]{Lemma} 
 \theoremstyle{remark}
 \theoremstyle{remark}
 \newtheorem{rem}[thm]{Remark}
 \theoremstyle{definition}

\theoremstyle{plain}  
\newtheorem{setup}[thm]{Setup}
\theoremstyle{plain}

\theoremstyle{definition}
\newtheorem{defi}[thm]{Definition}

\newcommand{\C}{{\mathbb{C}}}

\newcommand{\Q}{{\mathbb{Q}}}
\newcommand{\R}{{\mathbb{R}}}

\newcommand{\cA}{{\mathcal{A}}}

\newcommand{\cC}{{\mathcal{C}}}

\newcommand{\cE}{{\mathcal{E}}}
\newcommand{\cF}{{\mathcal{F}}}
\newcommand{\cG}{{\mathcal{G}}}

\newcommand{\cO}{{\mathcal{O}}}

\newcommand{\cU}{{\mathcal{U}}}
\newcommand{\cV}{{\mathcal{V}}}

\def\1{\mathbf{1}}

\newcommand{\ol}{\overline}

\newcommand{\Xo}{X^\circ}

\newcommand{\wX}{\widehat{X}}

\newcommand{\om}{\omega}

\newcommand{\wF}{\widehat \cF}

\newcommand{\PH}{\mathrm{PH}}

\newcommand{\Xr}{X_{\rm reg}}
\newcommand{\Xs}{X_{\rm sing}}

\newcommand{\ep}{\varepsilon}

\newcommand{\tor}{\mathrm{tor}}

\newcommand{\Tr}{\mathrm{Tr}}
\newcommand{\Id}{\mathrm{Id}}

\newcommand{\ddbar}{\d\dbar}
\newcommand{\wh}{\widehat}
\renewcommand{\d}{\partial}
\newcommand{\dbar}{\overline{\partial}}

\renewcommand{\ge}{\geqslant}
\renewcommand{\le}{\leqslant}

\newcommand{\End}{\operatorname{End}}

\newcommand{\supp}{\operatorname{supp}}

\title[Orbifold Chern classes in singular varieties]{A complex analytic approach to orbifold Chern classes on singular
varieties and its applications}

\date{\today}

\begin{abstract}
In this article, we prove the orbifold version of the Bogomolov-Gieseker inequality for stable $\Q$-sheaves on Kähler varieties, generalizing our earlier work \cite{GP25} in dimension three. We also provide a characterization of the equality case, a new purely analytical proof of the numerical characterization of complex torus quotients as well as a novel, complex analytic interpretation of the second orbifold Chern class associated to a $\Q$-sheaf.  
\end{abstract}

\author{Henri Guenancia}
\address{Univ. Bordeaux, CNRS, Bordeaux INP, IMB, UMR 5251, F-33400 Talence, France}

\email{henri.guenancia@math.cnrs.fr}

\author{Mihai P\u{a}un}
\address{Universit\"at Bayreuth, Mathematisches Institut, Lehrstuhl Mathematik VIII, Germany}
\email{mihai.paun@uni-bayreuth.de}

\begin{document}

\maketitle
 
 \tableofcontents
 
 \section{Introduction}\label{Intro}
 
 \subsection{Main results}
The present article contains general versions of the main results established in our previous work \cite{GP25}
so it can be seen as an addendum to it. Our main result is the following singular version of Bogomolov-Gieseker inequality for stable bundles. 

\begin{bigthm}
\label{thmA}
Let $(X,\omega)$ be a compact Kähler space of dimension $n$ and let $\cF$ be a coherent reflexive sheaf on $X$ of rank $r$ which is $[\omega]$-stable. Assume that there exists a closed analytic subset $Z\subset X$ of codimension at least three such that $X\setminus Z$ has at most quotient singularities and $\cF|_{X\setminus Z}$ has a structure of orbibundle. We have
\[\left(2rc_2(\cF)-(r-1)c_1(\cF)^2\right) \cdot [\omega]^{n-2} \ge 0,\]
and equality holds if and only if $\cF|_{\Xr}$ is locally free and projectively flat. 
\end{bigthm}

If $X$ is projective and $[\omega]$ is the first Chern class of an ample line bundle, then Theorem~\ref{thmA} can be deduced from the original Bogomolov-Gieseker inequality (and the characterization of the equality case) using Mehta-Ramanathan theorem \cite{MR84} and the singular version of Lefschetz hyperplane theorem \cite{GM88}.  In the analytic setting, Theorem~\ref{thmA} was proved for threefolds in \cite{GP25} and it was obtained in full generality in the very recent article \cite{FuOu} as consequence of two main techniques: the existence of a partial desingularisation for $X$ \cite{Ou,KO}, together with the orbifold version of important developments in MA theory (see especially \cite{GPSS22}) due to J. Scheffler in \cite{Scheffler}. The inequality part (without the characterization of the equality case) has been previously obtained in \cite{Ou,Ou2} and \cite{ZZZ} relying on \cite{Ou, KO}. \\

One of the motivations for establishing the theorem above was to obtain a purely analytical proof of the following.
\begin{bigcoro}
\label{uni4}
Let $X$ be a compact Kähler variety of dimension $n$ with log terminal singularities and $c_1(X)=0 \in H^2(X, \mathbb R)$. The following are equivalent:
\begin{enumerate}
\item There exists a quasi-étale cover $T\to X$ where $T$ is a complex torus. 
\item There exists a Kähler class $\alpha\in H^2(X,\R)$ such that $c_2(X)\cdot \alpha^{n-2}=0$. 
\end{enumerate}
\end{bigcoro}

The above result dates back to Yau's resolution of the Calabi conjecture \cite{Yau78} when $X$ is smooth. When $X$ is projective and $\alpha$ is the first Chern class of an ample line bundle, the result was proved by \cite{GKP} if $X$ is smooth in codimension two and later \cite{LuTaji} in full generality. The extension to the general Kähler case was proved by \cite{CGG} if $X$ is smooth in codimension two and then by \cite{CGG24} in full generality. However, the latter papers both rely on the algebraic case and the decomposition theorem \cite{BGL}. \\

\subsection{Orbifold Chern numbers : how to compute them}
In this short paragraph, we will briefly explain what are the Chern classes appearing in Theorem~\ref{thmA} and Corollary~\ref{uni4} above and how one can compute them thanks to a new characterization we provide. We refer to Section~\ref{OCC} for further details. 

Since $X$ is singular and $\cF$ need not be locally free, there exist several non-equivalent ways to define Chern classes or Chern numbers associated with $\cF$. In case $Z= \emptyset$, then $X$ is an orbifold, and the Chern-Weil theory applied to $\cF$ allows one to define its orbifold Chern classes. In case $X$ is projective, one can still define $c_1^2(\cF), c_2(\cF)$ seen as multilinear forms on $\mathrm{NS}(X)_{\R}$ using generic hyperplane sections, cf. \cite{MR717614}. In general, they can be defined as elements of the dual space $H^{2n-4}(X,\R)^*$ of the singular cohomology group of $X$ in degree $2n-4$; this was the approach followed by Graf-Kirschner in \cite{GK20}. Our contribution to the subject here is to give an interpretation of $c_1^2(\cF), c_2(\cF)$ via the Dolbeault cohomology of $X$ and relying on the Andreotti-Grauert theorem. More precisely, we prove the following result which might be of independent interest, cf Propositions~\ref{factor}-\ref{well posed} for the Dolbeault version below and Proposition~\ref{factor dR} and the discussion below it for the de Rham version. 

\begin{bigthm}
\label{thm urbi et orbi}
Let $(X,\omega)$ be a compact Kähler variety of dimension $n$ such that there exists a closed analytic subset $Z\subset X$ of codimension at least three such that $X\setminus Z$ has at most quotient singularities. Let $\cF$ be a coherent, torsion-free sheaf on $X$ such that $X\setminus Z$ is an orbibundle. 

\noindent
There exists a decomposition 
\[\omega^{n-2}=\Omega+\dbar \gamma\] 
where $\Omega$ is a smooth $(n-2,n-2)$-form supported away from $Z$ and $\gamma$ is a smooth $(n-2, n-3)$-form. Moreover, one has equality
\[c_2(\cF)\cdot \{\omega\}^{n-2}=\int_{X\setminus Z} c_2(\cF,h) \wedge \Omega\]
for any orbifold hermitian metric $h$ on $\cF|_{X\setminus Z}$. 
\end{bigthm}

The motivation for the result above is to have a definition for which the properties of the said Chern classes can be analysed by methods of geometric analysis. This result is crucially needed for the proof of Theorem~\ref{thmA}. Note that its content is more or less trivial when $n=3$ as in that case $Z$ is a finite collection of points and $\omega$ is then $dd^c$-exact near $Z$ so that $\Omega$ can be found "by hand", as explained e.g. in \cite{GK20} or \cite{GP25}. 

To put things into perspective, Graf and Kirschner use some "abstract" dimensional considerations to show that the class $\{\omega\}^{n-2} \in H^{2n-4}(X, \R)$ can be naturally represented by a singular chain $\sigma$ supported away from $Z$. Now, the de Rham-Weil isomorphism allows  to associate to $\sigma$ an orbifold form $\Omega'$ of degree $2n-4$ supported away from $Z$. There was, however, no clear relation exhibited between the {\it differential forms} $\omega^{n-2}$ and $\Omega'$ in \cite{GK20}, but this is now achieved thanks the de Rham version of Theorem~\ref{thm urbi et orbi}, i.e. Proposition~\ref{factor dR}.

\subsection{What goes on in the proof of Theorem~\ref{thmA}}
If $X$ is smooth and $\cF$ is locally free, the $[\omega]$-stability of $\cF$ coupled with Donaldson-Uhlenbeck-Yau theorem guarantees the existence of a Hermite-Einstein metric $h_{\cF}$ with respect to $\omega$. A classical computation shows that one has 
\[(2rc_2(\cF, h_{\cF})-(r-1)c_1^2(\cF, h_{\cF}))\wedge \omega^{n-2}\ge 0\] 
pointwise, and one can then integrate over $X$ to get the expected result. 

When $X$ and $\cF$ are singular, one can still find a Hermite-Einstein metric $h_{\cF}$ with respect to $\omega$ on $X_{\rm reg}\setminus Z$ thanks to \cite{C++}, hence the above inequality will hold pointwise on that set. Now, two problems arise before one can get any further: 
\begin{enumerate}[label=$\alph*.$]
\item $h_{\cF}$ need not be an orbifold hermitian metric on $X\setminus Z$. 
\item Even though the integral $\int_{X_{\rm reg}\setminus Z} c_2(\cF,h_{\cF}) \wedge \omega^{n-2}$ converges, it need not compute the orbifold Chern number $c_2(\cF)\cdot \{\omega\}^{n-2}$ because of the first item and the fact that $\mathrm{supp}(\omega^{n-2})=X$.  
\end{enumerate}

\medskip
In order to deal with $a.$, we pick a small open neighborhood $U\supset Z$ of $Z$ in $X$ and we now choose $\omega$ to be a metric on $X$ which has orbifold singularities in $X\setminus U$ and which is moreover smooth in the usual sense near $Z$. Such an ambient metric may look artificial, but in some sense it is "the best" we can get if $Z$ is non-empty. Moreover, the results in \cite{C++} still apply and yield a hermitian metric $h_{\cF}$ which is an orbifold metric on $X\setminus U$. \\

The second issue, $b.$, is much thornier as we now explain. We have a well-defined curvature form $\Theta(\cF, h_\cF)$ on $X_{\rm reg}\setminus Z$, which is orbifold smooth on $X\setminus U$. If one unravels the characterization of orbifold Chern numbers as provided by Theorem~\ref{thm urbi et orbi}, we see that for the equality
\[c_2(\cF)\cdot \{\omega\}^{n-2}=\int_{X_{\rm reg}\setminus Z} c_2(\cF, h_{\cF})\wedge \omega^{n-2}\]
to hold (from which Theorem~\ref{thmA} would follow), we need to have the following Stokes formula
\begin{equation}\label{intro1}
\int_{X_{\rm reg}\setminus Z}\Theta\wedge \dbar \alpha= 0
\end{equation}
where \[\Theta:= \Tr\big(\Theta(\cF, h_\cF)\wedge \Theta(\cF, h_\cF)\big),\] 
and $\alpha$ is a smooth form on $\Xr$ which is smooth near $Z$ and bounded with respect to our hybrid metric $\omega$. Surely, \eqref{intro1} will hold if $\supp(\alpha)\subset X\setminus U$. However, this condition will never be met by the specific $\alpha$ we want to plug in \eqref{intro1}. A good chunk of this paper is devoted to proving \eqref{intro1}, see Theorem~\ref{closed}. This relies in part on ideas already developed in \cite{GP25} (see Section~\ref{sec GP}) and in part on new arguments which are expanded in Section~\ref{sec closed}.

\subsection*{Acknowledgments} We would like to thank Benoît Claudon for enlightening discussions about Proposition~\ref{factor dR}. We also warmly thank J\'anos Koll\'ar for generously sharing his insights on a wide range of topics in this article.

\subsection*{Funding}This material is partially based upon work supported by the National Science Foundation under Grant No. DMS-1928930 while H.G. was in residence at the Simons Laufer Mathematical 
Science Institute (former MSRI) in Berkeley, California, during the Fall 2024 semester. H.G. is partially supported by the French Agence Nationale de la Recherche (ANR) under reference ANR-21-CE40-0010 (KARMAPOLIS). M.P. is gratefully acknowledging support from the Deutsche Forschungsgemeinschaft (DFG).

 \section{Notations and setup}\label{Noset}
Throughout the text, we will work in the following context. \\

$\bullet$ {\it Varieties}. Let $(X,\omega_X)$ be a compact normal Kähler variety of dimension $n\ge 2$. This means that $\omega_X$ is a Kähler form on the regular locus $\Xr$ which, under any local embedding of $X$ into the euclidean space, extends to a Kähler form there.  
Assume furthermore that there exists a closed analytic subset $Z\subset X$ such that 
\begin{itemize}
\item $X\setminus Z$ has at most quotient singularities,
\item $Z$ has codimension at least three.
\end{itemize}

\medskip

$\bullet$ {\it Sheaves}. Given $X$ as above, we say that a coherent, torsion-free sheaf $\cF$ of rank $r$ on $X$ is a $\Q$-sheaf if $\cF$ is reflexive and $\cF|_{X\setminus Z}$ is an orbibundle, maybe up to enlarging the analytic $Z$ introduced above (as long as its codimension does not exceed three). From now on, we will assume that 
\begin{itemize}
\item$\cF$ is a $\Q$-sheaf,
\item $\cF$ is $\omega_X$-stable,
\end{itemize} 
 
Let us recall the terminology used here. That $\cF$ is reflexive means that the natural map $\cF \to \cF^{**}$ is an isomorphism. Next, since $X\setminus Z$ has quotient singularities, any point $x\in X\setminus Z$ admits a {\it uniformizing chart}. That is, there exists a euclidean neighborhood $U$ of $x$ along with a finite Galois cover $p:V\to U$ such that $V\simeq B_{\mathrm C^n}(0,1)$. One says that $\cF|_{X\setminus Z}$ is an {\it orbibundle} if for any $x\in X\setminus Z$ and any uniformizing chart $p$ near $x$, the sheaf $\cE:=p^{[*]}\cF:=p^*(\cF)^{**}$ is locally free. 

\smallskip

We fix once and for all a proper, bimeromorphic surjective map $\pi: \wX\to X$ such that $\wX$ is smooth and $\wF:=\pi^*\cF/{\rm tor}$ is locally free. Such a map $\pi$ exists thanks to Hironaka's resolution of singularities and a result of Rossi \cite{Rossi68}.  The slope of $\cF$ with respect to $\omega_X$ is defined as $\mu_{\omega_X}(\cF)=\frac{1}{r} c_1(\wF)\cdot [\pi^*\omega_X]^{n-1}$. It is easy to check that the latter quantity does not depend on the chosen map $\pi$ satisfying the two required properties. Moreover the dependence of the slope on $\omega_X$ is only through its cohomology class $\kappa(\omega_X) \in H^1(X,\mathrm{PH}_X)$, cf \eqref{kappa}. Finally, we say that $\cF$ is $\omega_X$-stable if for any coherent subsheaf $\cG$ of $\cF$ with $\mathrm{rank}(\cG)\neq 0, r$, we have $\mu_{\omega_X}(\cG)<\mu_{\omega_X}(\cF)$. We refer to e.g. \cite[Section~2]{CGG} for further details.

\medskip

$\bullet$ {\it Metrics}. Let $\theta: X\to [0,1]$ be a cutoff function of $Z\subset X$. That is, we have $\theta \equiv 1$ away from $Z$ and $\mathrm{Supp}(\theta)\subset \Xo=X\setminus Z$. Fix two open sets $W'\subset W \Subset \Xo$ such that $W'\subset (\theta=1)$ and $\mathrm{Supp}(\theta)\Subset W$. We denote by $\cU$ a small neighborhood of $X\setminus W$. It is standard  (cf e.g. \cite[Lemma~11]{CGG24}) to construct a continuous function $\varphi_{\rm orb}:W\to \mathbb R$ such that $\omega_{\rm orb}:=\omega_X+dd^c \varphi_{\rm orb}$ is an orbifold Kähler metric on $W$. More precisely, it means that ${\omega_{\rm orb}}|_{W_{\rm reg}}$ is a Kähler metric and that $p^*\omega_{\rm orb}|_{U\cap W_{\rm reg}}$ extends to a Kähler metric on $V\cap p^{-1}(U\cap W)$ for any uniformizing chart $p$. We then choose $\ep_0>0$ small enough and consider 
\begin{equation}
\label{omega}
\omega:=\omega_X+dd^c(\ep_0\theta \varphi_{\rm orb}).
\end{equation}
By construction, $\omega$ is a Kähler metric near $Z$ and an orbifold Kähler metric on $W$. Moreover, it is not difficult to check that there exists $C>0$ such that
\begin{equation}
\label{qiso}
C^{-1}(\omega_X+\theta \omega_{\rm orb})\le \omega \le C(\omega_X+\theta \omega_{\rm orb}).
\end{equation} 

By \cite{C++}, there exists a smooth hermitian metric $h_{\cF}$ on $\cF|_{\Xr\setminus Z}$ such that its Chern curvature $i\Theta(\cF, h_{\cF})$ satisfies
\begin{equation}
\label{HYM}
\Lambda_\omega i\Theta(\cF, h_{\cF})  = \mu \mathrm{Id}_{\cF} \quad \mbox{on } \Xr \setminus Z,
\end{equation}
for some $\mu \in \mathbb R$. Moreover, $h_\cF$ extends to a orbifold smooth hermitian metric on $W$. That is, for any uniformizing chart $p:V\to U$, the hermitian metric $p^*h_{\cF}|_{U_0\cap W}$ on $\cE$ extends smoothly to $V\cap p^{-1}(U\cap W)$.  Finally, we have 
\begin{equation}
\label{L2}
\int_{\Xr \setminus Z} |\Theta(\cF, h_{\cF})|^2_{h_{\cF}, \omega}\,  \omega^n <+\infty.
\end{equation}

\section{Recollection from \cite{GP25}$+\ep$}
\label{sec GP}

\noindent In this section -and the following one- our main goal is to study the integrability properties of the Chern-Weil representatives of the first two
Chern classes of a coherent, reflexive sheaf $\cF$. The metrics we will consider on the ambient space $X$ and on the sheaf $\cF$ are required to satisfy certain properties we are next listing. As the title of the current section suggests, we intend to follow the approach in 
\cite{GP25} almost \emph{ad literam}: we will simply expose the general outline of the method, and show that it can be easily implemented in our current context. 
\begin{enumerate}

\item[(a)] Let $\omega_X$ be a non-singular Kähler metric on $X$. By using the notations in Section \ref{Noset}, we have a metric 
$\omega$ with the properties \eqref{omega} and \eqref{qiso}.

\item[(b)] The local resolutions of $\cF$ together with a partition of unity provide a reference metric $h_{\cF, 0}$ for the vector bundle $\cF|_{\Xr\setminus Z}$. It has a few interesting properties, most notably it induces a smooth metric on the desingularisation of $\cF$.

\item[(c)] Given a positive constant $C> 0$, we consider the class of metrics $h_\cF:= h_{\cF, 0}\exp(s)$ on $\cF|_{\Xr\setminus Z}$ with orbifold singularities on $W$, such that the following inequalities
\[\sup_{\Xr\setminus Z}|\Lambda_{\omega}\Theta(\cF, h_\cF)|_{h_\cF}\leq C, \qquad \int_{\Xr\setminus Z}|\Theta(\cF, h_\cF)|^2_{h_\cF, \omega}dV_{\omega}\leq C\]
together with 
\[\Tr \exp(s)\leq C,\qquad \exp(s)\geq e^{C\phi}\Id\]
hold. Notice that this class is non-empty, at least if $C\gg 0$, in case of an $\{\omega_X\}$-stable sheaf $\cF$, cf. \cite{GP25}. In the inequality above, we denote by $\phi$ a quasi-psh function with log-poles on $X$, such that $\{\phi= -\infty\}=\Xs\cup Z$ coincides with the singular loci of $X$ and $\cF$.
\end{enumerate}

\begin{rem} Let $X$ be a compact Kähler space and let $\cF$ be a coherent, torsion-free sheaf on $X$ as in Section \ref{Noset}. It would be very interesting to construct a metric $h_\cF$ on $\cF$ satisfying the properties listed at the point (c) above. A good candidate would be the metric $h_{\cF, 0}$ mentioned at (b), but it is unclear if the first $L^\infty$ property holds in this case. 
\end{rem}

\medskip

\noindent An important technical point is that in the complement of the set $Z$, the space $X$ has quotient singularities. More precisely, up to enlarging $Z$, one can assume that for any point $x\in X\setminus Z$ there exists an open subset $U\ni x$ such that 
\begin{equation}\label{eq-1}
(U, x)\simeq (\C^{n-2}, 0)\times (S, s),
\end{equation}
where $(S,s)$ is the germ of a surface with quotient singularities, see e.g. beginning of \textsection~9.1 in \cite{Joyce00}.  We can assume that $U$ is contained in a uniformising chart.
\medskip

\noindent Another notion which comes into play is the $\cC^\infty_X$-module $\cE^p$, defined as follows. For any open set $U\subset X$,  $\cE^p(U)$ consists of $p$-forms $\alpha\in C^{\infty}(U_{\rm reg},\mathcal A_X^p)$ satisfying 
\begin{enumerate}
\item $\alpha$ extends smoothly to a neighborhood of $U\cap Z$, 
\item $\sup_{U\cap \Xr} (|\alpha|_{\omega}+|d\alpha|_{\omega})<+\infty$.
\end{enumerate}
Note that $\alpha$ and $d\alpha$ are automatically bounded with respect to $\omega$ near $Z$ thanks to the first condition. Finally, we denote by $\cE^p_0\subset \cE^p$ the submodule of forms in $\cE^p$ supported away from $Z$. In other words, they are smooth $p$-forms $\alpha$ on $\Xr$ supported away from $Z$ such that $\alpha$ and $d\alpha$ are globally bounded with respect to $\omega$. 
\medskip

\noindent In this context, we have the following result.
\begin{thm}
\label{dbar closed}
Set $\Theta:=\mathrm{Tr}(\Theta(\cF, h_{\cF})\wedge \Theta(\cF, h_{\cF}))$. Then for any $\alpha \in \cE^{2n-5}_0(X)$, we have
\[\int_{\Xr} \Theta \wedge d \alpha=0.\]
\end{thm}

\noindent
Notice that, since $\Theta$ is real of type $(2,2)$, this is equivalent to proving that the equality \[\int_{\Xr} \Theta \wedge \dbar \alpha=0\] 
holds true for any $\alpha \in \cE^{n-2,n-3}_0(X)$. Note that $\Theta$ is a well-defined current on $X$ thanks to the second property in $(c)$ above, as $\omega$ dominates $\omega_X$. Theorem~\ref{dbar closed} implies that $d\Theta=0$ on $X\setminus Z$, but it is slighlty stronger since we allow more exact test forms than just the smooth ones. The proof is explained in a certain number of steps, detailed in the coming subsections.

\subsection{Cutoff functions and a Harnack-type result} As consequence of the local structure \eqref{eq-1} near an arbitrary point of $X\setminus Z$ we can construct very useful cutoff functions as in \cite[Lemma~5.9]{GP25}.

\begin{lem}
\label{cutoff}
Let $x\in X\setminus Z$ be an arbitrary point. Then, there exist $\ep_0>0$, a neighborhood $U$ of $x$ and a family of cut-off functions $(\rho_\delta)$ for $U_{\rm sing}= X_{\rm sing}\cap U$ such that 
\[\limsup_{\delta \to 0} \int_{U} \big( |\nabla \rho_\delta |_{\omega_{}}^{2+\ep_0}+|\Delta_{\omega} \rho_\delta |^{1+\ep_0}\big) \omega^n =0.\]
\end{lem}
\medskip

\noindent The arguments involved in the proof of \cite[Lemma~5.9]{GP25} are completely general, in particular they do not 
depend of the fact that $\dim X= 3$ in \emph{loc. cit.} The same comment applies to the following Harnack-type result, cf   \cite[Corollary~5.11]{GP25}.

\begin{cor}
\label{LI estimates theta}
Let $x\in X\setminus Z$. Then, there exist a neighborhood $U$ of $x$ and a number $p>1$ satisfying the following. For any non-negative function $f\in \mathcal C^2(U_{\rm reg})$ whose support is relatively compact in $U$, one has
\[f+(\Delta_{\omega} f)_- \in L^p(U_{\rm reg},\om^n) \quad \Longrightarrow \quad f\in L^{\infty}(U_{\rm reg}).\]
\end{cor}

\subsection{$C^0$ estimates for $h_{\cF}$ on $X\setminus Z$}
Let $h_\cF$ be a metric on $\cF$ such that the properties in (c) above are satisfied for some constant $C> 0$.
In what follows the notations are as above, and let $x\in X\setminus Z$ be a point at which $X$ has at worse quotient singularities. 
Consider $p:V\to U$ a uniformizing chart near $x$. We fix a smooth hermitian metric $h_{0, \cE}$ on $\cE$ and denote by $h_{\cE}$ the hermitian metric on $\cE|_{V_0}$ induced by $p^\star h_{\cF}|_{U_0}$. 
\medskip

\noindent The following statement shows that the two metrics we are considering here are quasi-isometric.
\begin{prop}
\label{h orbi}
With the above notation, there exists $C>0$ such that 
\[C^{-1} h_{0, \cE}\le h_{\cE} \le C\, h_{0, \cE} \quad \mbox{on }V_0.\]
\end{prop}
\medskip

\noindent The complete proof of this result is given in \cite{GP25}, cf. Theorem 5.15. 
We nevertheless review here next the main steps.
\begin{enumerate}

\item One easily obtains a weaker inequality, namely 
\[e^{N\phi} h_{0, \cE}\le h_{\cE} \le e^{-N\phi} h_{0, \cE} \quad \mbox{on }V_0.\]
holds for some large enough $N\gg 0$. This is first proved for the pull-back of the reference metric, and then the corresponding 
inequality for $h_{\cF}$ follows from (c).

\item Next, we can assume that $\cE$ is trivial, and let $\tau$ be any holomorphic section of $\cE^{\star}$. The main part of the proof consists in obtaining a differential inequality for the Laplacian (with respect to the pull-back of $\omega$) of $\chi \psi^\ep$, where $\chi$ is a cutoff function, $\psi:= |\tau|^2$ is the norm of our section and $0<\ep\ll 1$ is a small enough positive real.
 
\item Corollary \ref{LI estimates theta} is now used --remark that by choosing the $\ep$ above small enough, the integrability condition will be satisfied--, and this is the end of the proof. 
\end{enumerate}

\subsection{End of the proof: integration by parts on $X\setminus Z$} We will adopt the same strategy as in the previous subsection, by highlighting the main steps of the proof of the analogous statement Theorem 5.17 in \cite{GP25}, and adding the necessary explanations for the higher dimensional case.
\begin{enumerate}
\smallskip

\item[(i)] It is clear that the result we are after is local, so that we can assume that $\alpha$ has support in a set $U$ as in \eqref{eq-1} and Section \ref{Noset}. Moreover, we can assume that on the domain $V$ of the uniformising chart 
\[p: V\to U\]
we can construct a
trivialising frame $(e_1, \ldots, e_r)$ of $\cE\to V$.
Denote by $A$ the connection $1$-form on $V_0$ associated to $h_\cE$, with respect to the chosen basis. The curvature 
$\Theta(\cE, h_\cE)= \dbar A$ is the antiholomorphic derivative of $A$ over $V_0$. 

\smallskip
\noindent
Next, let $(\chi_\ep)$ be any family of cut-off functions for $U\setminus U_0=U\cap \Xs\cap Z$. The statement to be proved is equivalent to showing that the next equality is true
\[
\lim_{\ep\to 0}\int_{U}\Tr\big(\Theta(\cF, h_\cF)\wedge \Theta(\cF, h_\cF)\big)\wedge \dbar\chi_\ep\wedge \alpha= 0.
\]
By pulling back to $V$ via the map $p$ and using integration by parts, this is the same as showing that the equality
\[
\lim_{\ep\to 0}\int_{V}\Tr\big(A\wedge \Theta(\cE, h_\cE)\big)\wedge p^\star(\dbar\chi_\ep\wedge \dbar \alpha)=0
\]
holds. Now, we denote by $\omega_V:= p^\star\omega$ the pull-back of our metric via $p$ and set 
\[\psi_\ep:=|p^\star(\dbar\chi_\ep)|^2_{\omega_V}.\]
Cauchy-Schwarz inequality combined with the fact that 
the following holds
\begin{equation}\label{nfolds7}
|A|^2_{h_\cE, \omega_V}\leq C\sum_i |D'e_i|^2_{h_\cE, \omega_V}\end{equation}
(this is a consequence of Proposition \ref{h orbi})
show that it would be sufficient to prove that the relation
\begin{equation}\label{3folds7}
\limsup_{\ep\to 0}\int_{V}(\psi_\ep + |\partial \psi_\ep|+|\Delta''_{\omega_V} \psi_\ep|) \, \omega_V^n<\infty.
\end{equation}
Here we use a differential inequality for the Laplacian of $|e_i|^2_{h_\cE}$ (needed in order to bound the RHS of the inequality \eqref{nfolds7} above) and integration by parts.  
\smallskip

\item[(ii)] Recall that $(U,x)\simeq (\mathbb C^{n-2},0) \times (S,s)$ where $i:(S,s)\hookrightarrow (\mathbb C^N,0)$ is a germ of surface with log terminal singularities. Let $q:(\mathbb C^2, 0) \to (S,s)$ be a uniformizing chart and let $f:=i\circ q:\mathbb C^2\to \mathbb C^N$ be the composition. Therefore, up to shrinking $U$, the ramified cover $p:V\to U$ can be written as follows 
\[
(Z, u, v)\to \big(Z, f_1(u, v), \ldots,f_N(u, v)\big)
\]
where $Z= (z_1,\dots, z_{n-2})$. 
\smallskip 

\noindent This "explicit" expression is used first to define a function
\[\rho(z, u, v):= \theta(|u|^2+ |v|^2)+ \sum_{i=1}^N|f_i(u, v)|^2\]
and then the following truncation function
\[
\chi_\ep:= \Xi_\ep\big(\log\log(1/\rho)\big),
\]
which appears at the previous point. Here $\Xi_\ep: \mathbb R_+\to \mathbb R_+$ is equal to 1 on the interval $[0, \ep^{-1}[$ and with zero on $[1+\ep^{-1}, \infty[$. We will also assume that the absolute value of the first and second derivatives of $\Xi_\ep$ are uniformly bounded. 
\smallskip

\noindent The reason for the choice of $\rho$ as above is that we obviously have
\begin{align}\nonumber
\omega_V \simeq & \, dd^c\rho+ \omega_{\rm euc} \nonumber \\
\simeq & \,\theta dd^c(|u|^2+ |v|^2)+ 
dd^c\big(\sum_{i=1}^N|f_i(u, v)|^2\big)+ \omega_{\rm euc}
\nonumber 
\end{align}
where $\omega_{\rm euc}:= \sqrt{-1}\sum_{i= 1}^{n-2}dz_i\wedge d\ol z_i$. In particular, we see that the equations $(u=v=0)$ and $(f_1=\ldots=f_N=0)$ defining the bad set $p^{-1}(U_{\rm sing})$ also appear naturally in the potential of $\omega_V$; this plays a crucial role in $(iii)$ below through the use of \cite[Lemma~5.19]{GP25}. Let us finally observe that we have
\[\omega_V^n\simeq \Big(\theta dd^c(|u|^2+ |v|^2)+ 
dd^c\big(\sum_{i=1}^N|f_i(u, v)|^2\big)\Big)^2\wedge \omega_{\rm euc}^{n-2}=: dV_\theta.\]
\smallskip

\item[(iii)] Lemma~5.18 in \cite{GP25} establishes that \eqref{3folds7} holds as soon as we have 
\begin{equation}\label{nfolds}
\lim_{\ep\to 0}\int_{\supp \Xi_\ep'}\frac{dV_\theta}{\rho^2 \log^2{(|u|^2+|v|^2)}}< \infty.
\end{equation}
Finally,  \eqref{nfolds} is fully established in \cite{GP25} as a consequence of Lemma~5.19 therein and the discussion in Step 2 of the proof of Theorem 5.17. 
\end{enumerate}

\section{Closedness of $\Theta$ on $X$}
\label{sec closed}
In this section, we generalize Theorem~\ref{dbar closed} by showing that the second Chern form associated to $(\cF, h_{\cF})$ has stronger closedness properties: they hold globally on $X$ rather than in the complement of $Z$ only. 
\begin{thm}
\label{closed}
Let $\Theta$ be as Theorem~\ref{dbar closed} and let $\alpha\in \cE^{2n-5}(X)$. Then 
\[\int_{\Xr\setminus Z} \Theta \wedge d \alpha=0.\]
In particular, we have $d\Theta=0$ on $X$. 
\end{thm}

\begin{proof}
Let us stratify $Z$ as 
\[\emptyset = Z_{-1}\subset Z_0 \subset Z_1 \subset \ldots \subset Z_{n-4}\subset Z_{n-3}=Z\]
in such that way that for each index $0\le k \le n-3$, the set $Z_k$ is a closed analytic subset of $Z$ of dimension $\dim Z_k\le k$ and $Z_k\setminus Z_{k-1}$ is smooth.  We introduce the property $(P_k)$ as follows
\begin{equation}
\label{Pk}
\tag{$P_k$}
\forall S\in\cE^{2n-5}(X), \forall f \in C^{\infty}_0(X\setminus Z_k), \quad \int_{\Xr} f\Theta \wedge d S=-\int_{\Xr} d f \wedge \Theta \wedge S.
\end{equation}
The theorem will be proved if one can show that $(P_{-1})$ holds. Moreover, we remark that $(P_{n-3})$ holds thanks to Theorem~\eqref{dbar closed}. We are going to show by induction on $k$ that $(P_k)$ holds for any $k$. Assume that $(P_k)$ holds for some $k\le n-3$ and let us show that $(P_{k-1})$ holds as well.  We fix a form $S\in \cE^{2n-5}(X)$ and a  function $f \in C^{\infty}_0(X\setminus Z_{k-1})$ and we want to show that 
\begin{equation}
\label{goal}
\int_{\Xr} f\Theta \wedge d S=-\int_{\Xr} d f \wedge \Theta \wedge S.
\end{equation}
We proceed in two steps. \\

$\bullet$ First, since $Z_k\setminus Z_{k-1}$ is smooth, we can cover $Z_k \cap \mathrm{Supp}(f)$ by finitely many open sets $U_\alpha \hookrightarrow \mathbb C^N$ such that $Z_k\cap U_\alpha$ is a linear subspace when viewed inside $\mathbb C^N$. Let $\tau_\alpha$ be a partition of unity subordinate to the covering $(U_\alpha)$; by writing $S=\sum \tau_\alpha S$, we see that it is enough to show \eqref{goal} when $S$ has support in $U_\alpha$, which we will assume from now on. In particular, up to shrinking $U_\alpha$ further, there is no loss of generality assuming that $S$ is smooth form, i.e. $S\in \mathcal C^{\infty}_c(U_\alpha, \mathcal A^{2n-5}_X)$. 

As explained a few lines above, we can assume that the $k$-dimensional set $Z_k$ is given by the equations $z^i= 0$, for $i= k+1,\dots, N$. We have
\[S= \sum a_{I\ol J}(z)dz^I\wedge dz^{\ol J}\]
where $a_{I\ol J}$ are smooth, and $|I|+|J|=2n-5$. In particular, we have $\max\{|I|, |J|\}\ge n-2$. Since $k\leq n-3$, any subset of $\{1,\dots, N\}$ of cardinality at least $n-2$ contains at least one element belonging to $\{k+1,\dots, N\}$. Thus, $S$ is a sum of forms 
\[a_{I\ol J}(z)dz^{\ell}\wedge dz^{I'}\wedge dz^{\ol J} \quad \mbox{or} \quad a_{I\ol J}(z)dz^{\ol \ell}\wedge dz^{I}\wedge dz^{\ol J'}\]
for some $\ell \ge k+1$; i.e. such a form contains at least a differential normal to $Z$. Furthermore, we can write
\[a_{I\ol J}(z)dz^{\ell}\wedge dz^{I'}\wedge dz^{\ol J}= d\big(a_{I\ol J}(z) z^{\ell}dz^{I'}\wedge dz^{\ol J}\big)- z^{\ell}da_{I\ol J}(z)\wedge dz^{I'}\wedge dz^{\ol J}.\]
In conclusion, one can write 
\begin{equation}
\label{additional van}
S=S_0+dT_0
\end{equation}
where $T_0$ has degree $2n-6$ and the coefficients of both $S_0$ and $T_0$ are smooth (hence $S_0, T_0$ are in $\cE^{\bullet}(U_\alpha)$) belong to the ideal generated by the equations of $(Z, 0)$. 

\bigskip

 $\bullet$ Next, let $(\rho_{\ep})$ be a family of cutoff functions with support in $X\setminus Z_k$ such that $\rho_\ep \to 1_{X\setminus {Z_{k}}}$ as $\ep \to 0$. Since $f\rho_\ep \in C^{\infty}_0(X\setminus Z_k)$ and $(P_k)$ holds, we have 
\begin{eqnarray*}
\int_{\Xr} f\Theta \wedge d S&=&\lim_{\ep \to 0} \int_{\Xr} f\rho_\ep\Theta \wedge d S\\
&=&-\lim_{\ep \to 0} \Big(\int_{\Xr} \rho_\ep d f \wedge \Theta \wedge  S+\int_{\Xr}  fd \rho_\ep \wedge \Theta \wedge  S\Big)
\end{eqnarray*}
Since $\lim_{\ep \to 0} \int_{\Xr} \rho_\ep d f \wedge \Theta \wedge  S= \int_{\Xr} d f \wedge \Theta \wedge  S$, showing the desired identity \eqref{goal} is equivalent to showing  
\begin{equation}
\label{limit}
\lim_{\ep \to 0} \int_{\Xr}  f d \rho_\ep \wedge \Theta \wedge  S=0.
\end{equation}
Now we use the decomposition $S=S_0+dT_0$ from \eqref{additional van} where the coefficients of $S_0$ and $T_0$ vanish along $Z_k$, hence the coefficients of
\[d \rho_\ep \wedge S_0 \quad \mbox{and} \quad  d \rho_\ep \wedge T_0\]
 are uniformly bounded as $\ep \to 0$. Since $\Theta$ has $L^1$ coefficients and the volume of $\mathrm{Supp}(d \rho_\ep)$ (say with respect to $\omega$) shrinks to zero, it follows that 
 \[ \lim_{\ep \to 0} \int_{\Xr}  f d \rho_\ep \wedge \Theta \wedge  S_0=0.\]
Therefore, proving \eqref{limit} comes down to proving 
\begin{equation}
\label{limit2}
\lim_{\ep \to 0} \int_{\Xr}  f d \rho_\ep \wedge \Theta \wedge  dT_0=0.
\end{equation}
For that, we write
\begin{equation}
\label{dfg}
f d \rho_\ep \wedge  dT_0=d f \wedge d\rho_\ep \wedge T_0 -d\big(fd\rho_\ep \wedge T_0\big)
\end{equation}
Now, for $0<\delta \ll 1$, we have 
\begin{eqnarray*}
\int_{\Xr} \Theta\wedge d\big(fd\rho_\ep \wedge T_0\big) &= &\int_{\Xr} \rho_\delta\Theta\wedge d\big(fd\rho_\ep \wedge T_0\big)\\
&=&-\int_{\Xr}   \Theta \wedge (d\rho_\delta \wedge  fd\rho_\ep \wedge T_0)
\end{eqnarray*}
where we used $(P_k)$ in the last identity. Since $d \rho_\delta \wedge T_0$ has uniformly bounded coefficients, the latter integral goes to zero as $\delta\to 0$, hence it vanishes identically. Using \eqref{dfg}, one gets
\[ \int_{\Xr}  f d \rho_\ep \wedge \Theta \wedge  dT_0=\int_{\Xr} \Theta\wedge d f \wedge d\rho_\ep \wedge T_0.\] 
and the integral in the RHS above converges to zero as $\ep\to 0$ by the same token as for the previous integral. This shows \eqref{limit2} and finishes the proof of the theorem. 
\end{proof}


\section{Orbifold Chern classes revisited} 
\label{OCC}

In this section, we are aiming at a new characterization for the second orbifold Chern class associated to a reflexive sheaf on a compact normal variety which has only quotient singularities in codimension two. Our main motivation is to have a notion which is aligned with the methods and tools from geometric analysis. This will eventually prove to be crucial in order to obtain the 
Bogomolov-Gieseker inequality in this context.  
\smallskip

We start by discussing several cohomology theories (singular, de Rham, Dolbeault...) and the relations between them.

 \label{cohomology}
\subsection{The Leray map}
Let $X$ be a normal connected complex space of dimension $n$, not necessarily compact. Recall that a smooth differential form $\alpha$ on $X$ of degree $k$ (resp. bi-degree $(p,q)$) is by definition a smooth $k$-form (resp. $(p,q)$-form) on $\Xr$ which extends smoothly to local embedding in euclidean space. We say that $\alpha$ is $d$-closed (resp. $\dbar$-closed) if $d\alpha=0$ on $\Xr$ (resp. $\dbar \alpha=0$ on $\Xr$). We denote by $\cA^k$ (resp.  $\cA_X^{p,q}$) the $\cC^{\infty}_X$-module of smooth $k$-forms (resp. $(p,q)$-forms), and by $\Omega_X^p$ the coherent sheaf of Kähler differentials of degree $p$. In general, the complexes
 \begin{equation}
 \label{complex1}
  0\longrightarrow \mathbb R_X \longrightarrow \cA_X^{0} \overset{d}{\longrightarrow}\cA_X^{1} \overset{d}{\longrightarrow}\cA_X^{2}  \overset{d}{\longrightarrow} \cdots  \overset{d}{\longrightarrow} \cA_X^{2n} \overset{d}{\longrightarrow} 0
 \end{equation}
 
 \begin{equation}
 \label{complex2}
 0\longrightarrow \Omega_X^p\longrightarrow \cA_X^{p,0} \overset{\dbar}{\longrightarrow}\cA_X^{p,1} \overset{\dbar}{\longrightarrow}\cA_X^{p,2}  \overset{\dbar}{\longrightarrow} \cdots  \overset{\dbar}{\longrightarrow} \cA_X^{p,n} \overset{\dbar}{\longrightarrow} 0
 \end{equation}
 need not be exact, so that the singular cohomology space $H^k(X,\mathbb R)$  may not coincide with the de Rham cohomology space  
 \[H_{d}^{k}(X):=\frac{\{u\in \cC^\infty(X,\cA_X^{k}); \,\,d u=0\}}{\{d v; \,\,v\in \cC^\infty(X,\cA_X^{k-1})\}}.\] Similarly, the coherent cohomology space $H^q(X, \Omega_X^p)$) need not coincide with the  Dolbeault cohomology space \[H_{\dbar}^{p,q}(X):=\frac{\{u\in \cC^\infty(X,\cA_X^{p,q}); \,\,\dbar u=0\}}{\{\dbar v; \,\,v\in \cC^\infty(X,\cA_X^{p,q-1})\}}.\] 

 \medskip
 However, if we work with locally $\ddbar$-exact forms, we can go around this difficulty as we explain now. Let $\omega$ be a smooth real $(1,1)$-form which is locally $\ddbar$-exact. For each $p\geq 1$ we recall next the construction of the classes 
$\{\omega\}^p$ and  $[\omega]^p$ in $H^{2p}(X,\mathbb R)$ and $H^p(X, \Omega_X^p)$, respectively. 
\smallskip
 
By our assumption there exists a covering $X=\bigcup U_i$ by Stein open sets such that $\omega|_{U_i}=\sqrt{-1}\, \dbar \d \phi_i$ for some real-valued functions $\phi_i \in \cC^{\infty}(U_i)$.  In other words, the datum of $\omega$ is equivalent to the datum of an element in $H^0(X, \cC^\infty_X/\PH_X)$ where $\PH_X$ is the sheaf of pluriharmonic functions. From the exact sequence 
 \[0\longrightarrow  \PH_X \longrightarrow\cC^\infty\longrightarrow \cC^\infty/\PH_X\longrightarrow 0 \]
 we can associate to $\omega$ an element 
 \begin{equation}
 \label{kappa}
 \kappa(\omega)\in H^1(X, \PH_X).
 \end{equation} Next, we have the exact sequence
  \[0\longrightarrow  \mathbb R_X \longrightarrow\cO_X \overset{2\mathrm{Im}}{\longrightarrow}\PH_X\longrightarrow 0 \]
 yielding a connecting morphism $\delta^1:H^1(X, \PH_X)\to H^2(X,\mathbb R)$. We define
 \begin{equation}
 \label{de Rham class}
 \{\omega\}=:\delta^1(\kappa(\omega)) \in H^2(X,\mathbb R).
 \end{equation}
 Next, the differential map
 \[d:\cO_X\to \Omega_X^1\]
  factors through $\cO_X/\mathbb R_X\simeq \PH_X$, hence induces a map $d:\PH_X\to \Omega_X^1$. That map also induces a map in cohomology $d^1:H^1(X, \PH_X)\to H^1(X,\Omega_X^1)$, hence we get an element 
   \begin{equation}
 \label{coherent class}
 [\omega]=d^1(\kappa(\omega))\in H^1(X,\Omega_X^1).
 \end{equation} 
 In \v{C}ech cohomology, we have
 \[\kappa(\omega)=(\phi_{ij}), \quad \{\omega\}=(f_{ijk}), \quad \mbox{and} \quad [\omega]=(df_{ij}),\] 
 if we write $\phi_{ij}:=\phi_i-\phi_j$, $2\mathrm{Im}(f_{ij})=\phi_{ij}$ and $f_{ijk}:=f_{ij}+f_{jk}+f_{ki}$. For any integer $1\le p \le n$, we also get  cohomology classes 
 \[\{\omega\}^p:=\{\omega\} \cup \ldots \cup \{\omega\} \in H^{2p}(X,\mathbb R), \quad \mbox{and} \quad [\omega]^p:=[\omega] \cup \ldots \cup [\omega] \in H^p(X, \Omega_X^p)\]
 which are represented respectively by the $2p$-cocycle $\prod_{k=0}^{p-1}f_{i_{2k}i_{2k+1} i_{2k+2}}$ and the $p$-cocycle $df_{i_0i_1}\wedge df_{i_1i_2}\wedge \cdots \wedge df_{i_{p-1}i_p}$. 
 
 \medskip
  \subsubsection{From coherent to Dolbeault cohomology}

 Despite the failure of exactness of the complex \eqref{complex2}, we can still associate to any element $\beta\in H^q(X,\Omega_X^p)$ a $\dbar$-closed smooth $(p,q)$-form, via the usual Leray map. The notations in our next statement are self-explanatory.
 
 \begin{prop}
 \label{coh to dol}
Let $0\le p,q \le n$ be two positive integers. There exists a morphism
  \[\Phi=\Phi_{p,q}:H^q(X,\Omega_X^p)\longrightarrow H^{p,q}_{\dbar}(X)\]
  such that if $\omega$ is a locally $\ddbar$-exact smooth $(1,1)$-form on $X$, the following equality
\[\Phi_{p,p}([\omega]^p)= (-1)^{\frac{p(p-1)}{2}}[\omega^p]_{\dbar}\]
holds.
 \end{prop}
 
 \begin{proof}
 Pick $\beta \in H^q(X,\Omega_X^p)$ and think of $\beta$ as given by a $q$-cocycle of holomorphic $p$-forms $\beta^0=(\beta^0_{i_0\cdots i_q})$ on $U_{i_0\cdots i_q}$. We can construct iteratively a family 
 $(\beta^k)_{k=0, \ldots,q}$ such that 
 \begin{equation}
 \label{beta k}
 \beta^k\in C^{q-k}(X,\cA^{p,k}), \quad \delta \beta^k=0, \quad \dbar \beta^k=0,
 \end{equation}
 or, in other words, $\beta^k$ is a $(q-k)$-cocycle $\beta^k$ with values in closed $\dbar$-closed $(p,k)$-forms. We already have $\beta^0$ constructed; if $\beta^k$ is constructed for some $0\le k \le q-1$, we first pick 
 \[\widetilde \beta^k \in C^{q-k-1}(X, \cA^{p,k}) \quad \mbox{s.t.} \quad  \delta(\widetilde \beta^k)=\beta^{k}.\]
 For instance, if $(\rho_{\ell})$ is a partition of unity subordinate to our covering, we set $\widetilde \beta^k_{i_0\cdots i_{q-k-1}}:=\sum_{\ell} \rho_\ell \beta^1_{i_0\cdots i_{q-k-1}\ell}$. Then one defines
 \[\beta^{k+1}:=\dbar \widetilde\beta^k, \]
 which satisfies \eqref{beta k}. At the end of the process, i.e. when $k=q$, we obtain an element $\beta^q$ which is simply a $\dbar$-closed $(p,q)$-form. 
 
 \medskip
 
 Now we need to show that if we choose a different representative $\beta^0$ of $\beta$ or different primitives $\widetilde \beta^k$ of $\beta^{k}$, the resulting $\beta^q$ will differ only by the $\dbar$ of a smooth $(p,q-1)$-form.  Let us start with the first one. If we choose a different representative of $\beta^0$, this amount to replacing $\widetilde \beta^0$ by $\widetilde \beta^0+\widetilde \tau^0$ with $\tau^0$ holomorphic. In particular $\beta^1=\dbar \widetilde \beta^0$ does not change. 
 
 Now let us check that the choice of a primitive $\widetilde \beta^k$ of $\beta^{k}$ plays no role. At each step $0\le k \le q-2$, we chose an element $\widetilde \beta^k \in C^{q-k-1}(X, \cA^{p,k})$ such that $\delta(\widetilde \beta^k)=\beta^{k}$. Any other choice can be written $\widetilde \beta^k+\delta( \widetilde \gamma^k)$ where $\widetilde \gamma^k \in C^{q-k-2}(X, \cA^{p,k})$. Then $\beta^{k+1}$ becomes 
 \[\beta^{k+1}+\dbar \delta( \widetilde \gamma^k)=\delta(\widetilde \beta^{k+1}+\dbar \widetilde \gamma^k)=\delta(\widetilde \beta^{k+1}+\dbar \widetilde \gamma^k+\delta(\widetilde \gamma^{k+1})),\]
 so that $\beta^{k+2}$ becomes $\beta^{k+2}+\dbar \delta( \widetilde \gamma^{k+1})$, hence it is independent of the choice of $\widetilde \gamma^k$. 
 Once we arrive at the step $k=q-1$, we know that the previous choices (for $k\le q-2$) won't have any impact on the end result $\beta^q$, hence we only need to understand how $\beta^q$ depends on the possible choices at the level $q-1$. More precisely, our choices amount to replacing $\beta^{q-1}=\delta( \widetilde \beta^{q-1})$ by $\beta^{q-1}+\widetilde \gamma^{q-1}$, where this time $\widetilde \gamma^{q-1}$ is a $0$-cocycle with values in $\cA^{p,q-1}$, i.e. it is a well-defined $(p,q-1)$-form. This has the effect of replacing the closed $(p,q)$-form $\beta^q$ by $\beta^q+\dbar(\widetilde \gamma^{q-1})$, hence $[\beta^q]_{\dbar}:=\Phi_{p,q}(\beta)$ is well-defined.
 
 \medskip

 Let us now apply this construction to the class $\alpha:=[\omega]^p \in H^p(X,\Omega_X^p)$. We start with $p=1$. Using the previous notations, we have $\widetilde \beta^1_i=\sum_{\ell}\rho_\ell df_{i\ell}$. Since 
 \[\widetilde \beta^1_i-\widetilde \beta^1_j=df_{ij}=\sqrt{-1}\d \phi_{ij}=\sqrt{-1}\d\phi_i-\sqrt{-1}\d\phi_j,\]
 the $(1,0)$-forms $( \widetilde \beta^1_i-\sqrt{-1}\d\phi_i)$ on $U_i$ glue to a global $(1,0)$-form $u$ on $X$.  Moreover, we have
 \[\beta^1=\dbar \widetilde \beta^1=\omega+\dbar u.\]
 Let us now deal the case $p\ge 2$. Recall that $\beta^0=df_{i_0i_1}\wedge df_{i_1i_2}\wedge \cdots \wedge df_{i_{p-1}i_p}$ so that 
 \begin{equation*}
 \widetilde \beta^1_{i_0\cdots i_{p-1}}=df_{i_0i_1}\wedge df_{i_1i_2}\wedge \cdots \wedge df_{i_{p-2}i_{p-1}}\wedge (\sum_\ell\rho_\ell df_{i_{p-1}\ell})
 \end{equation*}
 hence
 \[\beta^1_{i_0\cdots i_{p-1}}=(-1)^{p-1}df_{i_0i_1}\wedge df_{i_1i_2}\wedge \cdots \wedge df_{i_{p-2}i_{p-1}}\wedge (\omega+\dbar u)\]
 thanks the previous case, where $u$ is a global $(1,0)$-form on $X$. Iterating, we get  \[(-1)^{\frac{p(p-1)}{2}}\alpha^p= (\omega+\dbar u)^p= \omega^p+\dbar \Big(\sum_{k=1}^p {p \choose k } u (\dbar u)^{k-1}\wedge \omega^{p-k}\Big),\]
 hence $\Phi([\omega]^p)= (-1)^{\frac{p(p-1)}{2}}[\omega^p]_{\dbar}$, as claimed.
  \end{proof}

  \subsubsection{From singular to de Rham cohomology}
 Similarly to the previous section, the Leray map allows one to associate to any singular cohomology class $\beta\in H^k(X,\mathbb R)$ a $d$-closed smooth $k$-form. 
 
 \begin{prop}
 \label{sing to dR}
 For any integer $0\le k  \le 2n$ there exists a morphism
\[\widehat \Phi: H^{k}(X, \R)\longrightarrow H^{k}_{d}(X)\]
such that if $k=2\ell$ and if $\omega$ is a locally $\ddbar$-exact smooth $(1,1)$-form on $X$, then the equality 
\[\widehat \Phi_{k}(\{\omega\}^\ell)= (-1)^{\frac{\ell(\ell- 1)}{2}}\{\omega^\ell\}_{d}\]
holds.  \end{prop}
 
 The proof of Proposition~\ref{sing to dR} is entirely similar to the one of Proposition~\ref{coh to dol} hence we will skip it. 
 
%


\subsection{Orbifold Chern numbers via de Rham cohomology}
In this subsection we introduce a new definition of orbifold Chern numbers relying on de Rham cohomology. This turns out to be equivalent to the one previously considered in \cite{MR717614} and \cite{GK20}, but as already mentioned before, our point of view turns out to be very useful in applications.

\subsubsection{A quick refresher on orbifolds} Recall that on a complex orbifold $M$, the natural smooth forms are those compatible with the orbifold structure. Namely, an orbifold smooth form on $M$ is a smooth differential form $\alpha$ on $M_{\rm reg}$ such that for any uniformizing chart $p:U\to V$, the form $p^*\alpha|_{U_{\rm reg}}$ extends smoothly to $V$. 

We can consider the orbifold de Rham and Dolbeault cohomology, denoted respectively by $H^{\bullet}_{d, {\rm orb}}(M)$ and $H^{\bullet}_{\dbar, {\rm orb}}(M)$, and similarly for the compactly supported version. Moreover, an orbibundle on $M$ is a vector bundle on $M_{\rm reg}$ such that for any uniformizing chart $p:U\to V$, the vector bundle $p^*E|_{U_{\rm reg}}$  admits a locally free extension to $V$. In other words, the reflexive pullback $p^{[*]}E|_U$ is locally free. One can define smooth orbifold hermitian metrics $h$ and $E$; their associated Chern forms are $d$-closed (and $\dbar$-closed, too) orbifold smooth forms which then induce cohomology classes $c_k^{\rm orb}(E)\in H^{2k}_{d, {\rm orb}}(M)$ (resp. $c_k^{\rm orb}(E)\in H^{k,k}_{\dbar, {\rm orb}}(M)$). We refer to \cite{Blache96} for further details.\\

\subsubsection{Definition of orbifold Chern numbers}

From now on, we fix the following setup: 

\begin{setup}
\label{setup}
Let $(X,\omega)$ be a normal, compact Kähler variety of dimension $n\ge 2$. Moreover, we assume that there exists a proper analytic subset $Z\subset X$ of codimension at least three such that $X\setminus Z$ has at most quotient singularities, i.e. $X\setminus Z$ is an orbifold. Finally, we let $E$ be an orbibundle on $X\setminus Z$. 
\end{setup}
  
 Denote by $i:Z\hookrightarrow X$ the closed immersion and by $j:X\setminus Z\to X$ the open immersion. The exact sequence 
 \[0\longrightarrow j_!j^*\R_X\longrightarrow \R_X\longrightarrow i_*i^*\R_X\longrightarrow 0\]
yields
\begin{equation}\label{chorb1}
0=H^{2n-5}(Z, \R)\to H^{2n-4}_{\rm c}(X\setminus Z, \R)\overset{\tau}{\to} H^{2n-4}(X, \R)\to H^{2n-5}(Z, \R)=0
\end{equation}
where the vanishing of the cohomology groups is a consequence of the fact that $\dim_{\mathbb R} Z\le 2n-6$. In particular, $\tau$ is an isomorphism. 
\smallskip

Now, there is a de Rham-Weil isomorphism
\[ \sigma:H^{2n-4}_{\rm c}(X\setminus Z, \R) \overset{\sim}{\longrightarrow} H^{2n-4}_{d,\rm { c, orb}}(X\setminus Z)\]
where the latter is the de Rham cohomology of orbifold smooth forms with compact support in $X\setminus Z$. For an orbifold bundle $E$ over $X\setminus Z$, Graf and Kirschner \cite{GK20} then define
\[\widetilde c_2(E)\cdot \{\omega\}^{n-2}:=c_2^{\rm orb}(E) \cdot (\sigma \circ \tau^{-1})(\{\omega\}^{n-2}) \in \mathbb R\]
where $c_2^{\rm orb}(E)\in H^{4}_{d, {\rm orb}}(X\setminus Z)$ is the usual second orbifold Chern class in orbifold de Rham cohomology, represented by the Chern form associated to any orbifold hermitian metric on $E$. At this point though, this definition is somehow impractical to work with from a differential geometric point of view.

\medskip

We now explain how to adapt our arguments above to give a more tractable definition of $\widetilde c_2(E)\cdot \{\omega\}^{n-2}$. The key input is the following result. 
\begin{prop}
\label{factor dR}
With the notation of Proposition~\ref{sing to dR}, the morphism $\widehat \Phi$ factors as follows
\[
\begin{tikzcd}
&H^{2n-4}_{d,{\rm  c}}(X\setminus Z) \arrow[d]\\
H^{2n-4}(X, \R)\arrow[r, "\widehat \Phi"] \arrow[ur, dashed, "\widehat \Psi"] & H^{2n-4}_{d}(X)
\end{tikzcd}
\]
In particular, given any Kähler form $\omega$ on $X$, one can write $\omega^{n-2}=\widehat \Omega+d\gamma$ where $\widehat \Omega$ is supported away from $Z$ and $\gamma$ is a smooth $(2n-5)$-form. 
\end{prop}

\begin{proof}
First, we claim that there exists a neighborhood $\cU$ of $Z$ in $X$ such that 
\begin{equation}
\label{vanishing 2}
H^{2n-5}(\cU,\R)=H^{2n-4}(\cU, \R)=0.
\end{equation}
Indeed,  we invoke the pair version, that is, (a semi-algebraic set, a closed subset) of {\L}ojasiewicz's theorem to triangulate $X$ in a way that $Z$ is a subcomplex, cf \cite{HironakaArcata} and references therein.  In particular, we can find a small neighborhood $\cU$ of $Z$ which is homotopy equivalent to $Z$, hence whose singular cohomology in degree greater than $\dim_{\mathbb R} Z$ vanishes.

\medskip

Now, let $\beta\in H^{2n-4}(X,\R)$ and let $\alpha$ be a representative of $\widehat \Phi(\beta)$. Since the formation of $\widehat \Phi$ commutes with restricting to open sets, we see that there exists a smooth $(2n-5)$-form $\gamma$ on $\cU$ such that $\alpha|_\cU=d\gamma$. Let $\chi$ be a smooth function with compact support in $\cU$ such that $\chi\equiv 1$ near $Z$, and let us set $\widehat \Omega=\alpha-d(\chi \gamma)$ which is supported away from $Z$. 

If in the above construction one picks another primitive of $\alpha|_{\cU}$, i.e. $\alpha|_\cU=d\gamma'$, then $\gamma-\gamma'$ arises as the image of a $(2n-5)$-cocycle on $\cU$ with values in $\R_\cU$ under the Leray map. Thanks to \eqref{vanishing 2}, we can find a smooth $(2n-6)$-form $\eta$ on $\cU$ such that $\gamma-\gamma'=d\eta$. In particular, if $\widehat \Omega':=\alpha-d(\chi' \gamma')$ for a possibly different cut-off function $\chi'$, then 
\[\widehat \Omega-\widehat \Omega'=d(\chi \gamma-\chi'\gamma')=d\Big((\chi-\chi')\gamma'-d\chi \wedge \eta\Big),\] 
and the form $(\chi-\chi')\gamma'-d\chi \wedge \eta$ is supported away from $Z$. This shows that the class $[\Omega]\in H^{2n-4}_{d,c}(X\setminus Z)$ only depends on $\beta$. Therefore one can unambiguously define $\widehat \Psi(\beta)=[\Omega]$, which indeed factors $\widehat \Phi$ since $\alpha=\Omega+d(\chi \gamma)$. 
\end{proof}

Recall that smooth forms on the orbifold $X\setminus Z$ are orbifold smooth. In particular, we have a natural map
\[r:H^{2n-4}_{d,{\rm  c}}(X\setminus Z)\to H^{2n-4}_{d, {\rm  c, orb}}(X\setminus Z).\]
We can now define the second orbifold Chern number as follows. 

\begin{defi}
\label{def c2 sing}
Let $E$ be an orbifold bundle on $X\setminus Z$ and let $\omega$ be a Kähler form on $X$. We set
\[c_2(E)\cdot \{\omega\}^{n-2}:= c_2^{\rm orb}(E) \cdot ( r\circ\widehat\Psi)(\{\omega\}^{n-2})\]
where $c_2^{\rm orb}(E)\in H^{4}_{d, {\rm orb}}(X\setminus Z)$ is the usual second orbifold Chern class in orbifold de Rham cohomology. 
\end{defi}

We can reformulate the above definition in the following more concrete way. Pick an orbifold hermitian metric $h$ on $E$ and write $\omega^{n-2}=\widehat \Omega+d\gamma$ where $\widehat \Omega$ is supported away from $Z$, which  is possible thanks to Proposition~\ref{factor dR}. Then we have
\begin{equation}
\label{formula}
c_2(E) \cdot \{\omega\}^{n-2}= \int_{X\setminus Z} c_2(E,h)\wedge \widehat \Omega,
\end{equation}
and the fact that the RHS is independent of any choice is guaranteed by Proposition~\ref{factor dR} asserting that the class of $\widehat \Omega$ in the relevant cohomology with compact support only depends on $\{\omega\}^{n-2}$.

\subsubsection{Comparison with previous definitions}

When $X$ is projective, $E$ is a reflexive sheaf on $X$ and  $L$ is an ample line bundle, then Mumford \cite{MR717614}  showed that one can attach to $E$ and $L$ rational numbers $\hat c_2(E)\cdot c_1(L)^{n-2}$ satisfying the following property. Pick general elements $H_i\in |mL|$ for $m\gg 1$ so that $S:=H_1\cap \cdots H_{n-2}\subset X$ is an orbifold surface and $E|_S$ is an orbibundle. Then we have $\hat c_2(E)\cdot c_1(L)^{n-2}=\frac{1}{m^{n-2}}c_2^{\rm orb}(E|_S)$. 

\begin{lem}
\label{lem Mum}
Let $(X,\omega)$ be a normal compact Kähler variety of dimension $n$ having at most quotient singularities in codimension two. Assume that $\{\omega\}\in H^2(X,\R)$ is the image of the first Chern class of an ample line bundle $L$. Then
\[\hat c_2(E)\cdot c_1(L)^{n-2}=c_2(E)\cdot \{\omega\}^{n-2}.\]
\end{lem}

\begin{proof}
Embed $X\subset \mathbb P^N$ in a way that $[\omega]=[\omega_{\rm FS}|_X]$. Moreover, up to scaling on can assume that there are divisors $H_1, \ldots, H_{n-2}\in |L|$ such that the complete intersection surface $S:=H_1\cap \ldots \cap H_{n-2}\cap X$ satisfies $S\cap Z=\emptyset$. It is not difficult (see e.g. \cite[Lemma~32]{CGG24}) to construct a family of explicit smooth $d$-closed $(n-2,n-2)$-form $(\Omega_\ep)$ in $\mathbb P^N$ such that 
 \begin{enumerate}
 \item $[\Omega_\ep]=[\omega_{\rm FS}^{n-2}]$ in $H^{n-2, n-2}_{d}(\mathbb P^N)$. 
 \item ${\Omega_\ep}|_X$ is supported in a neighborhood of $S$, hence away from $Z$.
 \item $\Omega_\ep \wedge [X] \to [S]$ weakly as $\ep \to 0$.  
 \end{enumerate}
 In particular, we see that for any hermitian orbifold bundle $(E,h)$ on $X\setminus Z$ and any $\ep>0$, one has 
\[c_2(E)\cdot \{\omega\}^{n-2}= \int_{X\setminus Z} c_2(E,h) \wedge \Omega_\ep,\]
hence
\[c_2(E)\cdot \{\omega\}^{n-2}=\lim_{\ep\to 0} \int_{X\setminus Z} c_2(E,h) \wedge \Omega_\ep=c_2^{\rm orb}(E|_S).\]
\end{proof}

Let us finish this section by showing that the orbifold Chern numbers as considered in Definition~\ref{def c2 sing} coincide with the ones considered by Graf-Kirschner. Since the latter were known to generalize Mumford's, Lemma~\ref{lem GK} generalizes Lemma~\ref{lem Mum}. 

\begin{lem}
\label{lem GK}
Let $E$ be an orbifold bundle on $X\setminus Z$ and let $\omega$ be a Kähler form on $X$. 
We have equality $\widetilde c_2(E)\cdot \{\omega\}^{n-2}=c_2(E)\cdot \{\omega\}^{n-2}$. 
\end{lem}

\begin{proof}
We have the following diagram
\[
\begin{tikzcd}
H^{2n-4}_{\rm c}(X\setminus Z, \R) \arrow [d, "\sigma"] \arrow[r, "\tau"] & H^{2n-4}(X, \R)  \arrow[d, "\widehat \Psi"]  \\
H^{2n-4}_{d, {\rm  c, orb}}(X\setminus Z, \R) &\arrow[l, "r"] H^{2n-4}_{d,{\rm  c}}(X\setminus Z)  
\end{tikzcd}
\]
  where $j$ is the natural maps induced by the inclusion of smooth forms in orbifold smooth forms, over $X\setminus Z$. We claim that it is a commutative diagram. Indeed, this is because the map $\widehat \Psi\circ \tau$ is the usual Leray map from singular cohomology to de Rham cohomology (with compact support). In other words, we have $r\circ \widehat \Psi=\sigma\circ \tau^{-1}$, hence $\widetilde c_2(E)\cdot \{\omega\}^{n-2}=c_2(E)\cdot  \{\omega\}^{n-2}$ as claimed. 
  \end{proof}


 \subsection{Orbifold Chern numbers via Dolbeault cohomology}
 \label{sec OCN}
  
In Setup~\ref{setup}, we can associate to the Kähler form $\omega$ its cohomology class $\kappa(\omega)\in H^1(X,\PH_X)$ and then either consider
\begin{itemize}
\item the singular cohomology class $\{\omega\}\in H^2(X,\R)$, or
\item the coherent cohomology class $[\omega]\in H^1(X,\Omega_X^1)$.
\end{itemize}
In the first case, Proposition~\ref{factor dR} allows us to associate to $\{\omega\}^{n-2}$ a de Rham cohomology class of degree $2n-4$ on $X$ with compact support in $X\setminus Z$. 

\medskip

Working with coherent cohomology rather than singular cohomology has a few advantages which we now exhibit. Since $\dim Z \le n-3$, it follows from \cite{Ohsawa84} (see also \cite[Theorem~2]{DemQconvex}) that $Z$ is strongly $(n-2)$-complete, i.e. it admits a smooth (exhaustion) function $\varphi$ such that $i\d \dbar \varphi$ has at most $n-3$ negative or zero eigenvalues. By \cite[Theorem~1]{DemQconvex}, this implies that $Z$ admits a fundamental family of strongly $(n-2)$-complete neighborhoods 
 \[Z\hookrightarrow \mathcal V\subset X.\]
  By Andreotti-Grauert's vanishing theorem \cite{AG62}, this implies that for any coherent sheaf $\cF$ on $X$, we have 
 \begin{equation}
 \label{vanishing}
 H^{n-2}(\mathcal V, \cF)=0.
 \end{equation}
which can be seen as the coherent analog of \eqref{vanishing 2}. As a consequence of Proposition~\ref{coh to dol}, we get the following statement which is a partial analog of Proposition~\ref{factor dR}.
\begin{prop}
\label{factor} 
%
%
Given any class $\beta \in H^{n-2}(X,\Omega_X^{n-2})$, the Dolbeault class $\Phi(\beta)\in H^{n-2,n-2}_{\dbar}(X)$ admits a smooth representative whose support is contained in the complement of $Z$.
In particular, there exists a decomposition 
\begin{equation}
\label{decomp}\omega^{n-2}=\Omega+\dbar \gamma
\end{equation}
 such that $\Omega$ is a smooth $(n-2,n-2)$-form such that $\supp (\Omega)\subset X\setminus Z$ and such that $\gamma$ is a smooth $(n-2,n-3)$-form.
\end{prop}

\begin{proof}
Choose a representative $\alpha\in \Phi(\beta)$. By \eqref{vanishing}, we have $\beta|_{\cV}=0$. Since $\Phi$ commutes with restrictions to an open set, we infer that $\Phi(\beta)|_{\mathcal V}=[\alpha|_{\mathcal V}]_{\dbar}=0$. This means that there exists a smooth $(n-2, n-3)$-form $\gamma_{\mathcal V}$ on $\mathcal V$ such that  
\begin{equation*}
\alpha|_{\mathcal V}=\dbar \gamma_{\mathcal V}.
 \end{equation*}
 Now, let $\chi$ be a cutoff function such that $\chi\equiv 1$ near $Z$ and $\mathrm{Supp}(\chi)\Subset \cV$, set $\gamma:=\chi \gamma_{\mathcal V}$. The $\dbar$-closed $(n-2, n-2)$-form 
 \begin{equation}
 \label{Omega}
 \Omega:=\alpha-\dbar \gamma
 \end{equation}
 is indeed compactly supported away from $Z$ and satisfies $[\alpha]_{\dbar}= [\Omega]_{\dbar}$. The last statement of the proposition follows by taking $\alpha=\omega^{n-2}$, which is legitimate by the second half of Proposition~\ref{coh to dol}. 
 \end{proof}

 At this point, one would like to argue that $[\Omega]\in H^{n-2,n-2}_{\dbar,c}(X\setminus Z)$ in Dolbeault cohomology with compact support only depends on $[\omega]^{n-2}\in H^{n-2}(X,\Omega_X^{n-2})$, just like in Proposition~\ref{factor dR}. For instance, it would be a consequence of the vanishing $H^{n-3}(\mathcal V, \Omega_{\mathcal V}^{n-2})=0$ (which e.g. holds if $Z$ has codimension at least four by Andreotti-Grauert vanishing theorem), but we do not know if the latter holds in full generality. In principle, this could be an obstruction to unambiguously define orbifold Chern numbers via Dolbeault cohomology. 
 
 However, we need something a bit weaker to properly define these numbers. Namely, one would need to know that given any hermitian metric $h$ on $E$ and any smooth $(n-2, n-3)$-form $\eta$ such that we have
 \[\supp (\dbar \eta)\subset \mathcal V\setminus Z,\]
then the equality
\begin{equation}
 \label{chorbi2}\int_{X\setminus Z} c_2(E,h)\wedge \dbar \eta=0\end{equation}
 holds. The problem here is that a-priori we don't know anything about
the support of $\eta$ itself, other that it is contained in $\mathcal V$. The above vanishing seems difficult to obtain by working directly on $X$ or on a resolution of $X$ where the pullback of $E$ becomes torsion-free modulo its torsion.
\smallskip

\noindent The equality \eqref{chorbi2} becomes elementary if one uses the existence of so-called orbifold partial resolutions \cite{KO}, 
\cite{Ou}, as well as \cite{Ko26},
which we now recall.  
 
 In the Setup~\ref{setup} and up to enlarging $Z$, there exists a compact orbifold $\widehat X$ and a surjective, holomorphic map $g: \widehat X\to X$ such that $g$ is an isomorphism over $X\setminus Z$ and such that $E$, viewed as an orbibundle over $X\setminus Z \simeq \widehat X \setminus \mathrm{Exc}(g)$, extends to an orbibundle $\widehat E$ on $\widehat X$. 
 Using the existence of $g$, one can prove:
 
 \begin{prop}
 \label{well posed}
 The integral
 \[\int_{X\setminus Z} c_2(E,h)\wedge \Omega\]
 is independent of the choice of an orbifold hermitian metric $h$ on $E$ and a form $\Omega$ as in \eqref{decomp}. Moreover, it coincides with $c_2(E)\cdot \{\omega\}^{n-2}$ from Definition~\ref{def c2 sing}.
 \end{prop}
 
 \begin{proof}
 First, if $h'$ is another orbifold hermitian metric on $E$, then $c_2(E,h)-c_2(E,h')=\dbar \alpha$ for some orbifold smooth $(4,3)$-form and $\int_{X\setminus Z}\dbar \alpha\wedge \Omega=0$ since $\dbar \Omega=0$. Next, assume that we have two decompositions $\omega^{n-2}=\Omega+\dbar \gamma = \Omega'+\dbar \gamma'$ as in \eqref{decomp} and set $\eta:=\gamma-\gamma'$. Next, choose $h$ such that $g^*h$ extends to an orbifold hermitian metric $\widehat h$ on $\widehat E$. Then we have 
 \begin{eqnarray*}
 \int_{X\setminus Z} c_2(E,h)\wedge (\Omega-\Omega')&=&\int_{\widehat X\setminus \mathrm{Exc}(g)} g^*(c_2(E,h)\wedge \dbar \eta)\\
 &=&\int_{\widehat X} c_2(\widehat E,\widehat h)\wedge \dbar (g^*\eta)
 \end{eqnarray*}
 which vanishes by Stokes formula. Here, we have used that $g^*\eta$ is a smooth form on the orbifold $\widehat X$, hence it is orbifold smooth too and Stokes formula can be legitimately applied.
 
 Similarly, if $\omega^{n-2}= \Omega'+d\gamma'$ as in \eqref{factor dR}, then 
  \begin{eqnarray*}
 \int_{X\setminus Z} c_2(E,h)\wedge (\Omega-\Omega')&=&\int_{\widehat X\setminus \mathrm{Exc}(g)} g^*(c_2(E,h)\wedge (\dbar \gamma-d \gamma')\\
 &=&\int_{\widehat X} c_2(\widehat E,\widehat h)\wedge \dbar (g^*\gamma)-\int_{\widehat X} c_2(\widehat E,\widehat h)\wedge d (g^*\gamma')
 \end{eqnarray*}
 which vanishes by Stokes formula, since $c_2(\widehat E,\widehat h)$ is an orbifold smooth $d$ and $\dbar$-closed form and both $g^*\gamma$ and $g^*\gamma'$ are smooth, hence orbifold smooth forms.  The proposition now follows from \eqref{formula}. 
 \end{proof}

\begin{rem}
As we see from the discussion above, one can very well define the orbifold Chern numbers in the context \ref{setup} via Dolbeault cohomology, by using the existence of the partial resolution of singularities. It would be very interesting to know if this can be done without invoking this result. The coholomogy class we are interested in, induced by $\dbar \eta$ in \eqref{chorbi2}, obviously belongs to the kernel of the map
\[H^{n-2,n-2}_{\dbar,c}(\mathcal V\setminus Z)\to H^{n-2,n-2}_{\dbar }(\mathcal V),\]
which does not seem to be easy to identify.
\end{rem}

\section{Bogomolov-Gieseker inequality}

\begin{thm}
\label{thm BG}
In our main setup, we have
\[(2rc_2(\cF) -(r-1)c_1(\cF)^2)\cdot [\omega_X]^{n-2} \ge 0,\]
and equality occurs if and only if $\cF|_{\Xr}$ is locally free and projectively unitary flat. 
\end{thm}

\begin{proof}
Let us start with the easy direction and assume that $\cF|_{\Xr}$ is locally free and projectively unitary flat. This means that there exists a smooth hermitian metric $h$ on $\cF|_{\Xr}$ such that its Chern curvature tensor satisfies $\Theta_h(\cF)=\alpha \mathrm{Id}_{\cF}$ for some $(1,1)$-form $\alpha$ on $\Xr$. It is not completely clear at this point that $h$ is an orbifold metric on $X\setminus Z$ and that it can be used to compute Chern numbers. Instead, we consider the induced hermitian vector bundle $(\mathrm{End}(\cF), \mathrm{End}(h))$ on $\Xr$, which is flat. Now, given an open subset $U\subset X\setminus Z$ admitting a uniformizing chart $p:V\to U$, the hermitian flat vector bundle $p^* (\mathrm{End}(\cF), \mathrm{End}(h))|_{U_{\rm reg}}$ on $V_0:=p^{-1}(U_{\rm reg})$ extends to a hermitian flat bundle on $V$. Indeed, $p^* (\mathrm{End}(\cF), \mathrm{End}(h))|_{U_{\rm reg}}$ is given by a representation $\pi_1(V_0)\to \mathrm{U}(r^2, \C)$ and the natural morphism $\pi_1(V_0)\to \pi_1(V)$ is an isomorphism given that the complement of $V\setminus V_0$ has codimension at least two in the smooth manifold $V$. This shows that $\mathrm{End}(\cF)$ can be endowed with a structure of hermitian flat orbibundle on $X\setminus Z$. In particular, we find 
\[c_2^{\rm orb}(\mathrm{End}(\cF|_{X\setminus Z}))=0\in H^{4}_{{\rm dR, orb}}(X\setminus Z).\] 
Now, the orbifold Chern classes of the orbibundle $\cF|_{X\setminus Z}$ satisfy 
$2rc_2^{\rm orb}(\cF|_{X\setminus Z}) -(r-1)c_1^{\rm orb}(\cF|_{X\setminus Z})^2=c_2^{\rm orb}(\mathrm{End}(\cF|_{X\setminus Z}))$, as one can see e.g. from the proposition in \cite[p23]{Blache96}. This concludes the proof of the first part of the theorem. 

\medskip

In the other direction, let $h_{\cF}$ be the singular hermitian metric on $\cF$ provided by \cite{C++} satisfying \eqref{HYM}. Set 
\[\Delta(\cF, h_{\cF}):=2rc_2(\cF,h_\cF) -(r-1)c_1(\cF,h_\cF)^2) \quad \mbox{on} \quad  \Xr\setminus Z.\] 
We have 
\begin{equation}
\label{semipositive}
\Delta(\cF, h_{\cF})\wedge \omega^{n-2}\ge0 \quad \mbox{on } \,\, \Xr\setminus Z,
\end{equation}
with equality if and only if $(\cF, h_\cF)$ is projectively unitary flat on $\Xr\setminus Z$, thanks to \eqref{HYM}. More precisely, since $h_{\cF}$ is HE with respect to $\omega$, the term $\Delta(\cF, h_{\cF})\wedge \omega^{n-2}$ coincides up to a dimensional constant with the squared norm of $\Theta(\cF, h_\cF)-\frac 1r \Tr\Theta(\cF, h_\cF) \cdot \mathrm{Id}_\cF$, see e.g. the proof of \cite[Theorem~4.4.7]{Koba}.

We saw in Proposition~\ref{factor} that one can write $\omega_X^{n-2}=\widehat\Omega+d \gamma$ with $\widehat \Omega, \gamma$ smooth forms and $\widehat\Omega$ compactly supported away from $Z$.  Without loss of generality, one can assume that $\mathrm{Supp}(\widehat\Omega)\subset W'$ so that $h_{\cF}$ is orbifold smooth on $\mathrm{Supp}(\widehat\Omega)$. In particular, we have 
\begin{equation}
\label{c2 eq}
(2rc_2(\cF) -(r-1)c_1(\cF)^2)\cdot [\omega_X]^{n-2} = \int_{X\setminus Z} \Delta(\cF, h_{\cF})\wedge \widehat\Omega
\end{equation}
 Now, recall from \eqref{omega} that $\omega=\omega_X+dd^c \varphi$. In particular, one can write 
\[\omega^{n-2}=\omega_X^{n-2}+d \tau, \quad \mbox{with} \,\, \tau:=d^c \varphi\wedge \sum_{k=0}^{n-3} \omega^k \wedge \omega_X^{n-3-k}.\]
From \eqref{qiso}, it is elementary to check that $d^c \varphi$ and $\omega_X$ are bounded with respect to $\omega$ and that $\tau$ vanishes near $Z$. In particular, if we set $\alpha:=\gamma+\tau$, then we have 
\begin{equation}
\label{alpha}
\omega^{n-2}=\widehat\Omega+d\alpha \quad \mbox{with} \quad \alpha\in \cE^{n-2, n-3}(X).
\end{equation}
 By Theorem~\ref{closed}, we infer
\begin{equation}
\label{vanishing Delta}
\int_{\Xr\setminus Z} \Delta(\cF, h_{\cF})\wedge d\alpha =0.
\end{equation}
Putting \eqref{c2 eq}, \eqref{alpha} and \eqref{vanishing Delta} together, we get
Therefore, 
\[(2rc_2(\cF) -(r-1)c_1(\cF)^2)\cdot [\omega_X]^{n-2} =\int_{\Xr\setminus Z} \Delta(\cF, h_{\cF})\wedge \om^{n-2}\]
and the first part of the theorem follows from \eqref{semipositive}.

In the equality case, we get that $\Theta(\cF, h_\cF)=\frac 1r \Tr\Theta(\cF, h_\cF) \cdot \mathrm{Id}_\cF$ on $\Xr\setminus Z$. That is, $(\cF, h_\cF)|_{\Xr\setminus Z}$ is projectively hermitian flat. The second part of the theorem can now be  obtained along the same lines as in the proof of Corollary~5.4 in \cite{GP25}.   
\end{proof}

\section{Application to uniformization}

In this section, we explain how to use the main result to deduce a numerical characterization of (possibly) singular complex torus quotients. The result was first proved in the projective case by \cite{GKP16, LuTaji} and then generalized to the Kähler setting \cite{CGG, CGG24} yet relying on the algebraic case and the Beauville-Bogomolov decomposition \cite{BGL}.  

\begin{thm}
Let $X$ be a compact Kähler variety of dimension $n$ with log terminal singularities and $c_1(X)=0 \in H^2(X, \mathbb R)$. The following are equivalent:
\begin{enumerate}
\item There exists a quasi-étale cover $T\to X$ where $T$ is a complex torus. 
\item There exists a Kähler class $\alpha$ such that $c_2(X)\cdot \alpha^{n-2}=0$. 
\end{enumerate}
\end{thm}

Just as in \cite{CGG24}, one can deduce from the theorem above its "pair version", that is,  characterizing quotients $T/G$ with $G$ a finite group of automorphism possibly having fixed points in codimension one, cf \cite[Corollary~7]{CGG24}. 
\begin{proof}[Sketch of proof]
We know from \cite{GSS} that $T_X$ is $\alpha$-polystable, hence has a decomposition $T_X=\bigoplus \cF_i$ into stable pieces with trivial first Chern class. From the BG inequality, we see that $c_2(\cF_i)\cdot \alpha^{n-2}=0$ for all indices $i$. The equality case implies that $\cF_i|_{\Xr}$ is unitary flat. By the analytic version of \cite{GKP16} (available thanks to relative MMP for projective morphisms  due to \cite{Fuj22}, \cite{LM22}, see also \cite{DHP}), there exists a finite quasi-étale Galois cover $p:Y\to X$ such that $p^{[*]}\cF_i$ is flat, hence locally free. Since this holds for all indices $i$ and $p^{[*]}T_X=T_Y$, we see that $T_Y$ is locally free. In particular, it follows from the resolution of Zariski-Lipman conjecture for klt varieties \cite{Dru14, GKKP} that $Y$ is smooth and satisfies $c_2(Y)\cdot p^*\alpha^{n-2}=0$. Thanks to Yau's theorem, $Y$ is a torus quotient. 
\end{proof}

 \bibliographystyle{smfalpha}
\bibliography{biblio}

\newcommand{\etalchar}[1]{$^{#1}$}
\providecommand{\bysame}{\leavevmode ---\ }
\providecommand{\og}{``}
\providecommand{\fg}{''}
\providecommand{\smfandname}{\&}
\providecommand{\smfedsname}{\'eds.}
\providecommand{\smfedname}{\'ed.}
\providecommand{\smfmastersthesisname}{M\'emoire}
\providecommand{\smfphdthesisname}{Th\`ese}
\begin{thebibliography}{CGN{\etalchar{+}}23}

\bibitem[AG62]{AG62}
{\scshape A.~Andreotti {\normalfont \smfandname} H.~Grauert} -- {\og
  Th{\'e}or{\`e}mes de finitude pour la cohomologie des espaces complexes\fg},
  \emph{Bull. Soc. Math. Fr.} \textbf{90} (1962), p.~193--259 (French).

\bibitem[BGL22]{BGL}
{\scshape B.~Bakker, H.~Guenancia {\normalfont \smfandname} C.~Lehn} -- {\og
  Algebraic approximation and the decomposition theorem for {K{\"a}hler}
  {Calabi}-{Yau} varieties\fg}, \emph{Invent. Math.} \textbf{228} (2022),
  no.~3, p.~1255--1308 (English).

\bibitem[Bla96]{Blache96}
{\scshape R.~Blache} -- {\og Chern classes and {H}irzebruch--{R}iemann--{R}och
  theorem for coherent sheaves on complex-projective orbifolds with isolated
  singularities\fg}, \emph{Math.~Z.} \textbf{222} (1996), no.~1, p.~7--57.

\bibitem[CGG22]{CGG}
{\scshape B.~Claudon, P.~Graf {\normalfont \smfandname} H.~Guenancia} -- {\og
  {Numerical characterization of complex torus quotients}\fg}, \emph{{Comment.
  Math. Helv.}} (2022), no.~4, p.~pp. 769–799.

\bibitem[CGG24]{CGG24}
\bysame , {\og Equality in the {Miyaoka}-{Yau} inequality and uniformization of
  non-positively curved klt pairs\fg}, \emph{C. R., Math., Acad. Sci. Paris}
  \textbf{362} (2024), no.~S1, p.~55--81 (English).

\bibitem[CGN{\etalchar{+}}23]{C++}
{\scshape J.~Cao, P.~Graf, P.~Naumann, M.~P{\u a}un, T.~Peternell {\normalfont
  \smfandname} X.~Wu} -- {\og {Hermite--Einstein metrics in singular
  settings}\fg}, Preprint
  \href{http://arxiv.org/abs/2303.08773}{arXiv:2303.08773}, 2023.

\bibitem[Dem90]{DemQconvex}
{\scshape J.-P. Demailly} -- {\og Cohomology of q-convex spaces in top
  degrees\fg}, \emph{Math. Z.} \textbf{204} (1990), no.~2, p.~283--295
  (English).

\bibitem[DHP24]{DHP}
{\scshape O.~Das, C.~Hacon {\normalfont \smfandname} M.~P{\u{a}}un} -- {\og On
  the 4-dimensional minimal model program for {K{\"a}hler} varieties\fg},
  \emph{Adv. Math.} \textbf{443} (2024), p.~68 (English), Id/No 109615.

\bibitem[Dru14]{Dru14}
{\scshape S.~Druel} -- {\og The {Z}ariski-{L}ipman conjecture for log canonical
  spaces\fg}, \emph{Bull. Lond. Math. Soc.} \textbf{46} (2014), no.~4,
  p.~827--835.

\bibitem[FO25]{FuOu}
{\scshape X.~Fu {\normalfont \smfandname} W.~Ou} -- {\og {Orbifold
  Bogomolov-Gieseker inequalities on compact Kähler varieties}\fg}, Preprint
  \href{https://arxiv.org/abs/2511.03530}{arXiv:2511.03530}, 2025.

\bibitem[Fuj22]{Fuj22}
{\scshape O.~Fujino} -- {\og {Minimal model program for projective morphisms
  between complex analytic spaces}\fg}, Preprint
  \href{http://arxiv.org/abs/2201.11315}{arXiv:2201.11315}, 2022.

\bibitem[GK20]{GK20}
{\scshape P.~Graf {\normalfont \smfandname} T.~Kirschner} -- {\og {Finite
  quotients of three-dimensional complex tori}\fg}, \emph{Ann.~Inst.~Fourier
  (Grenoble)} \textbf{70} (2020), no.~2, p.~881--914.

\bibitem[GKKP11]{GKKP}
{\scshape D.~Greb, S.~Kebekus, S.~J. Kov{\'a}cs {\normalfont \smfandname}
  T.~Peternell} -- {\og Differential forms on log canonical spaces\fg},
  \emph{Publ. Math. Inst. Hautes {\'E}tudes Sci.} (2011), no.~114, p.~87--169.

\bibitem[GKP16a]{GKP16}
{\scshape D.~Greb, S.~Kebekus {\normalfont \smfandname} T.~Peternell} -- {\og
  \'{E}tale fundamental groups of {K}awamata log terminal spaces, flat sheaves,
  and quotients of abelian varieties\fg}, \emph{Duke Math. J.} \textbf{165}
  (2016), no.~10, p.~1965--2004.

\bibitem[GKP16b]{GKP}
\bysame , {\og Singular spaces with trivial canonical class\fg}, in
  \emph{Minimal Models and Extremal Rays, Kyoto, 2011}, Adv. Stud. Pure Math.,
  vol.~70, Mathematical Society of Japan, Tokyo, 2016, Preprint
  \href{http://arxiv.org/abs/1110.5250}{arXiv:1110.5250}, p.~67--113.

\bibitem[GM88]{GM88}
{\scshape M.~Goresky {\normalfont \smfandname} R.~MacPherson} --
  \emph{Stratified {Morse} theory}, Ergeb. Math. Grenzgeb., 3. Folge, vol.~14,
  Berlin etc.: Springer-Verlag, 1988 (English).

\bibitem[GP25]{GP25}
{\scshape H.~Guenancia {\normalfont \smfandname} M.~P{\u{a}}un} -- {\og
  {Bogomolov-Gieseker inequality for log terminal {K{\"a}hler} threefolds}\fg},
  \emph{Commun. Pure Appl. Math.} \textbf{78} (2025), no.~11, p.~2206--2244
  (English).

\bibitem[GPSS24]{GPSS22}
{\scshape B.~Guo, D.~H. Phong, J.~Song {\normalfont \smfandname} J.~Sturm} --
  {\og Diameter estimates in {K{\"a}hler} geometry\fg}, \emph{Commun. Pure
  Appl. Math.} \textbf{77} (2024), no.~8, p.~3520--3556 (English).

\bibitem[Gue16]{GSS}
{\scshape H.~Guenancia} -- {\og {Semistability of the tangent sheaf of singular
  varieties}\fg}, \emph{Algebraic Geometry} \textbf{3} (2016), no.~5,
  p.~508--542.

\bibitem[Hir75]{HironakaArcata}
{\scshape H.~Hironaka} -- {\og Triangulations of algebraic sets\fg}, Algebraic
  {Geom}., {Proc}. {Symp}. {Pure} {Math}. 29, {Arcata} 1974, 165-185 (1975).,
  1975.

\bibitem[Joy00]{Joyce00}
{\scshape D.~D. Joyce} -- \emph{Compact manifolds with special holonomy},
  Oxford Math. Monogr., Oxford: Oxford University Press, 2000 (English).

\bibitem[KO25]{KO}
{\scshape J.~Koll{\'a}r {\normalfont \smfandname} W.~Ou} -- {\og {Orbifold
  modifications of complex analytic varieties}\fg}, Preprint
  \href{http://arxiv.org/abs/2512.20708}{arXiv:2512.20708}, 2025.

\bibitem[Kob87]{Koba}
{\scshape S.~Kobayashi} -- \emph{{Differential geometry of complex vector
  bundles.}}, {Princeton, NJ: Princeton University Press; Tokyo: Iwanami Shoten
  Publishers}, 1987 (English).

\bibitem[Kol26]{Ko26}
{\scshape J.~Kollár} -- {\og {Personal communication to the authors}\fg}, 2nd
  of January 2026.

\bibitem[LM22]{LM22}
{\scshape S.~Lyu {\normalfont \smfandname} T.~Murayama} -- {\og {The relative
  minimal model program for excellent algebraic spaces and analytic spaces in
  equal characteristic zero}\fg}, Preprint
  \href{http://arxiv.org/abs/2209.08732}{arXiv:2209.08732}, 2022.

\bibitem[LT18]{LuTaji}
{\scshape S.~Lu {\normalfont \smfandname} B.~Taji} -- {\og A characterization
  of finite quotients of abelian varieties\fg}, \emph{Int. Math. Res. Not.}
  \textbf{2018} (2018), no.~1, p.~292--319 (English).

\bibitem[MR84]{MR84}
{\scshape V.~B. Mehta {\normalfont \smfandname} A.~Ramanathan} -- {\og
  Restriction of stable sheaves and representations of the fundamental
  group\fg}, \emph{Invent. Math.} \textbf{77} (1984), p.~163--172 (English).

\bibitem[Mum83]{MR717614}
{\scshape D.~Mumford} -- {\og Towards an enumerative geometry of the moduli
  space of curves\fg}, in \emph{Arithmetic and geometry, {V}ol. {II}}, Progr.
  Math., vol.~36, Birkh{\"a}user Boston, Boston, MA, 1983, p.~271--328.

\bibitem[Ohs84]{Ohsawa84}
{\scshape T.~Ohsawa} -- {\og Completeness of noncompact analytic spaces\fg},
  \emph{Publ. Res. Inst. Math. Sci.} \textbf{20} (1984), p.~683--692 (English).

\bibitem[Ou24]{Ou}
{\scshape W.~Ou} -- {\og {Orbifold modifications of complex analytic
  varieties}\fg}, Preprint
  \href{http://arxiv.org/abs/2401.07273}{arXiv:2401.07273}, 2024.

\bibitem[Ou25]{Ou2}
\bysame , {\og {Orbifold Chern classes and Bogomolov-Gieseker
  inequalities}\fg}, Preprint
  \href{http://arxiv.org/abs/2512.22273}{arXiv:2512.22273}, 2025.

\bibitem[Ros68]{Rossi68}
{\scshape H.~Rossi} -- {\og Picard variety of an isolated singular point\fg},
  \emph{Rice Univ. Stud.} \textbf{54} (1968), no.~4, p.~63--73 (English).

\bibitem[Sch24]{Scheffler}
{\scshape J.~Scheffler} -- {\og {Auxiliary Monge-Ampère Equations in Orbifold
  Setting -- a Mean-Value Inequality}\fg}, Preprint
  \href{https://arxiv.org/abs/2404.02812}{arXiv:2404.02812}, 2024.

\bibitem[Yau78]{Yau78}
{\scshape S.-T. Yau} -- {\og {On the Ricci curvature of a compact K{\"a}hler
  manifold and the complex Monge-Amp{\`e}re equation. I.}\fg}, \emph{Commun.
  Pure Appl. Math.} \textbf{31} (1978), p.~339--411.

\bibitem[ZZZ25]{ZZZ}
{\scshape C.~Zhang, S.~Zhang {\normalfont \smfandname} X.~Zhang} -- {\og {The
  Miyaoka-Yau inequality for minimal Kähler klt spaces}\fg}, Preprint
  \href{http://arxiv.org/abs/2503.13365}{arXiv:2503.13365}, 2025.

\end{thebibliography}

\end{document}